\documentclass[12pt]{article}
\usepackage{amssymb,amsfonts,amsthm}
\usepackage{amsmath,amsthm,amssymb}
\usepackage{amscd}
\usepackage{amsfonts}
\usepackage{color}
\newcommand{\ocs}{\hbox{OccSet}}
\newcommand{\non}{\operatorname{Non}}
\newcommand{\lin}{\operatorname{Lin}}
\newcommand{\iso}{\operatorname{Isot}}
\newcommand{\var}{\operatorname{var}}

\newtheorem{theorem}{Theorem}[section]
\newtheorem{ex}[theorem]{Example}
\newtheorem{cor}[theorem]{Corollary}

\newtheorem{fact}[theorem]{Fact}
\newtheorem{lemma}[theorem]{Lemma}

\newtheorem{prop}[theorem]{Proposition}

\newtheorem{question}{Question}
\newtheorem{claim}{Claim}

\begin{document}

\title{Finitely based sets of 2-limited block-2-simple words.}
\author{ Olga Sapir }
\date{Communicated by Mikhail Volkov}

\maketitle

\begin{abstract}  Let $\mathfrak A$ be an alphabet and $W$ be a set of words in the free monoid ${\mathfrak A}^*$. Let $S(W)$ denote the Rees quotient  over the ideal of  ${\mathfrak A}^*$ consisting of all words that are not subwords of words in $W$. A set of words $W$ is called {\em finitely based}  if the monoid $S(W)$ is finitely based.

 A word $\bf u$ is called 2-limited if each variable occurs in $\bf u$ at most twice.
A {\em block} of a word $\bf u$ is a maximal subword of $\bf u$ that does not contain any linear variables.
We say that a word $\bf u$ is  {\em block-2-simple} if each block of $\bf u$ involves at most two distinct variables.
We provide an algorithm that recognizes finitely based sets of words among sets of 2-limited block-2-simple words.

We also present new sufficient conditions under which a set of words is non-finitely based.

\end{abstract}

\section{Introduction}

A semigroup is said to be {\em finitely based} (FB) if there is a finite subset of its identities from which all of its identities may be deduced.
Otherwise, a semigroup is said to be {\em non-finitely based} (NFB). Throughout this article, elements of a countably infinite alphabet $\mathfrak A$ are called {\em variables} and elements of the free monoid $\mathfrak A^*$  and free semigroup $\mathfrak A^+$ are called {\em words}. In 1969, Perkins \cite{P} found the first two examples of finite NFB semigroups.  One of these examples was the 25-element monoid obtained from the set of words
$W= \{abtba, atbab, abab, aat\}$ by using the following construction attributed to Dilworth.

 Let ${\mathfrak A}$ be an alphabet and $W$ be a set of words in the free monoid ${\mathfrak A}^*$. Let $S(W)$ denote the Rees quotient  over the ideal of  ${\mathfrak A}^*$ consisting of all words that are not subwords of words in $W$. For each set of words $W$, the semigroup $S(W)$ is a monoid with zero whose nonzero elements are the subwords of words in $W$. Evidently, $S(W)$ is finite if and only if $W$ is finite.


 We call a set of words $W$ {\em finitely based} if the monoid $S(W)$ is finitely based.
 In this paper we continue to study the following problem.

\begin{question}\cite[M. Sapir]{SV} \label{qMS} Given a set of words $W$ is there an algorithm that decides if $W$ is FB or NFB?
\end{question}

Articles \cite{ LZL, LL, OS, OS3} provide some partial answers to this question when $W$ consists of a single word. In particular,
article \cite{OS} contains an algorithm for selecting FB words among words over a two-letter alphabet.

\begin{fact} \label{2varone} \cite[Theorem 5.1]{OS}  A word $\bf u$ over a two-letter alphabet $\{a,b\}$ is FB if and only if modulo renaming variables, $\bf u$ is of the form $a^nb^m$ or $a^nba^m$ for some $n,m\ge 0$.
\end{fact}

 However, the task of describing all FB {\em sets} of words even over a two-letter alphabet is very challenging \cite{JS}.
A word $\bf u$ is called {\em $k$-limited} if each variable occurs in $\bf u$ at most $k$ times. Recent result of Jackson \cite{MJ2}
implies that  the problem of classifying all FB sets of 2-limited words is also very difficult.
  On the other hand, articles \cite{JS} and \cite{OS} readily imply the following description of all FB sets of 2-limited words over a two-letter alphabet.

\begin{fact} \label{2var} Let $W$ be a set of 2-limited words over a two-letter alphabet. Then $W$ is FB if and only if modulo renaming variables, $W$
contains $\{a^2b^2, abab, abba\}$ or  every word in $W$ is a subword of a word in $\{a^2b^2, aba\}$.
\end{fact}

 If a variable $t$ occurs exactly once in a word $\bf u$ then we say that $t$ is {\em linear} in $\bf u$. If a variable $x$ occurs
 more than once in ${\bf u}$ then we say that $x$ is {\em non-linear} in ${\bf u}$.
Recently, we succeeded to generalize Fact \ref{2varone} into an algorithm that selects FB words among words with at most two non-linear variables \cite{OS3}.

A {\em block} of $\bf u$ is a maximal subword of $\bf u$ that does not contain any linear variables of $\bf u$.
A word $\bf u$ is called  {\em block-n-simple} if each block of $\bf u$ involves at most n distinct variables. For example, the word
$ababt_1bccbcbt_2caa$ is block-2-simple. Evidently, every word with at most two non-linear variables is block-2-simple.
The main goal of this article is to generalize Fact \ref{2var} into an algorithm that selects FB sets of words among  sets of 2-limited block-2-simple words.

 We use $\var S$ to refer to the variety of semigroups generated by $S$.
When $\var S(W) = \var S(W')$ we write $W \sim W'$ and say that $W$ and $W'$ are {\em equationally equivalent} (abbreviated to {\em e.e.}).  Notice that two equationally equivalent sets of words might look quite different. For example, it is easy to check that $\{a^2b^2c^2, ded\} \sim \{a^2b^2\}$. In view of Lemma \ref{nlimited}(ii) below, Fact \ref{2var} can be reformulated as follows.

\begin{fact} \label{2varcor} (cf. Fact \ref{2var}) Let $W$ be a set of 2-limited words over a two-letter alphabet. Then $W$ is FB if and only if modulo renaming variables, $W \sim \{a^2b^2, abab, abba\}$ or $W \sim \{a^2b^2\}$ or $S(W)$ is contained in $\var S(\{a^2b, ab^2\})$.
\end{fact}

Recall that a monoid $M$ is said to be {\em hereditarily finitely based} (HFB) if every monoid subvariety of $\var M$ is finitely based.
If $W$ is a set of words such that the monoid $S(W)$ is HFB then we say that $W$ is HFB. HFB sets of words can be easily recognized
syntactically using Corollary 8.3 in \cite{OS3} (cf. Lemma \ref{hfb} below). In particular,  every set of block-1-simple words is HFB. Thus $\{a^2b, ab^2\}$ is HFB and Fact \ref{2varcor} can  be reformulated as follows.

\begin{fact} \label{2varcor1} (cf. Fact \ref{2varcor}) Let $W$ be a set of 2-limited words over a two-letter alphabet. Then $W$ is FB if and only if modulo renaming variables, $W \sim \{a^2b^2, abab, abba\}$ or $W \sim \{a^2b^2\}$ or $W$ is HFB.
\end{fact}

According to \cite[Corollary 3.8]{OS3}, every block-1-simple word is e.e. to a set of words with at most one non-linear variable. We do not know whether every block-2-simple word  is e.e. to a set of words with at most two non-linear variables or not.
The following theorem generalizes Fact \ref{2varcor1} and will be proved in Section \ref{sec:alg}.

\begin{theorem} \label{main1} (cf. Theorem \ref{main}) A set of 2-limited block-2-simple words $W$ is FB if and only if
one of the following is true:

(A) $W \sim \{a^2b^2, abab, abba\}$;

(B)  $W \sim \{a^2b^2\}$;

(C) $W$ is HFB;

(D) $\var S(\{abtab, abtba\}) \subseteq \var S(W) \subseteq \var S(\{abtab, abtba, atbba \})$;

(D$'$)  $\var S(\{abtab, abtba\}) \subseteq \var S(W) \subseteq \var S(\{abtab, abtba, abbta \})$.

\end{theorem}

The `if' part of Theorem \ref{main1} is an immediate consequence from Corollary 3.2 in \cite{JS}, Theorem 5.1 in \cite{OS} and Theorem 5.3 in \cite{OS2}.
 In Section \ref{sec:fb}, we provide simple algorithms for verifying when a set of words $W$ satisfies each of the five conditions (A)--(D$'$) in Theorem \ref{main1}.

In order to prove Theorem \ref{main1} we establish some sufficient conditions under which a set of words is NFB (see Theorems \ref{xyyx} and \ref{t1} below). The sufficient condition in  Theorem \ref{xyyx} has an especially simple formulation and it inspired us to find another similar sufficient condition in Section \ref{sec:sufcon}.

\begin{theorem} \label{xyyx1} (cf. Corollary \ref{sevenint}) A set of word $W$ is NFB provided it satisfies one of the following:

(i) (Theorem \ref{xyyx})  $\var S(W)$ contains $S(\{abba\})$  but does not contain $S(\{at_1abt_2b\})$;

(ii) (Corollary \ref{xytxy})  $\var S(W)$ contains  $S(\{abtab\})$  but does not contain $S(\{abtba\})$.
\end{theorem}

In Section \ref{sec:twovar} we establish some sufficient conditions under which a word is e.e. to a set of words with at most two non-linear variables. In particular, we show that every HFB set of words is e.e. to a set of words with at most two non-linear variables  (see Theorem \ref{eqeq} below).

\section{Preliminaries}\label{sec:iso}

\subsection{Isoterms and related concepts}

A word $\bf u$ is said to be an {\em isoterm} \cite{P} for a semigroup $S$ if $S$ does not satisfy any non-trivial identity of the form ${\bf u} \approx {\bf v}$.


\begin{lemma} \label{prec}  \cite[Lemma 3.3]{MJ}
Let $W$ be a set of words and $S$ be a monoid.
Then each word in $W$ is an isoterm for $S$ if and only if $\var(S)$ contains $S(W)$.
\end{lemma}

Lemma \ref{prec} allows  to introduce a quasi-order on sets of words in $\mathfrak A^*$ as follows.
If $W$ and $W'$ are two sets of words then we write $W \preceq W'$ if for any monoid $S$ each word in $W'$ is an isoterm for $S$ whenever  each word in $W$ is an isoterm for $S$.

The following proposition implies that if  $W \preceq W' \preceq W$ then  $W$ and $W'$ are equationally equivalent ($W \sim W'$).
It also implies that if we identify sets of words modulo $\sim$ then we obtain an ordered set antiisomorphic to
the set of all varieties of the form $\var S(W)$ ordered under inclusion.

\begin{prop} \label{pr1} \cite[Proposition 2.3]{OS3} For two sets of words $W$ and $W'$
the following conditions are equivalent:

(i)  $W \preceq W'$;

(ii) Each word in $W'$ is an isoterm for $S(W)$;

(iii) $\var S(W)$ contains  $S(W')$.

\end{prop}

The relations $\preceq$ and $\sim$ can be extended to individual words. For example, if $\bf u$ and $\bf v$ are two words then ${\bf u} \sim {\bf v}$ means $\{{\bf u}\} \sim \{{\bf v}\}$. Also, if $W$ is a set of words and $\bf u$ is a word then $W \preceq {\bf u}$  means $W \preceq \{{\bf u}\}$.

We use $W^{\le}$ to denote the closure of $W$ under taking subwords and $W^{\preceq}$ to denote the closure of $W$ under going up in order $\preceq$. It is easy to see that $W \subseteq W^{\le} \subseteq  W^{\preceq}$ and $W \sim W^{\le} \sim  W^{\preceq}$.

Given a set of identities $\Sigma$, we use $\Sigma^\delta$ to denote the closure of $\Sigma$ under deleting variables. For example,
$\{xytxy \approx yxtyx\}^\delta = \{xytxy \approx yxtyx, xyxy \approx yxyx\}$.
 We use $\iso(\Sigma)$ to denote the set of all words in ${\mathfrak A}^*$ that are isoterms for $\var \Sigma^\delta$ (the semigroup variety defined by $\Sigma^\delta$).
Using Lemma \ref{prec} it is easy to show that $W = \iso(\Sigma)$ is the largest subset of ${\mathfrak A}^*$ such that $S(W)$ is contained in $\var \Sigma^\delta$ (see Fact 8.1 in \cite{OS1}).

\subsection{Left and right sides of the identities  $\sigma_\mu$, $\sigma_1$ and $\sigma_2$ }

The following  identities
\[xt_1xyt_2y \approx xt_1yxt_2y,  \hskip.1in xyt_1xt_2y \approx yxt_1xt_2y, \hskip.1in  xt_1yt_2xy \approx xt_1yt_2yx\] we denote respectively by $\sigma_{\mu}$,  $\sigma_1$ and $\sigma_{2}$. Notice that the identities $\sigma_1$ and $\sigma_{2}$ are dual to each other. We use letter $t$ with or without subscripts to denote linear (1-occurring) variables. If we use letter $t$ several times in a word,  we assume that different occurrences of $t$ represent distinct linear variables; so $xtxytyt$ abbreviates $xt_1xyt_2yt_3$ for example.
Using this abbreviation, the identities $\sigma_\mu$, $\sigma_1$ and $\sigma_2$ we rewrite as follows:
\[xtxyty \approx xtyxty,  \hskip.1in xytxty \approx yxtxty, \hskip.1in  xtytxy \approx xtytyx.\]
The next fact says that the left and the right sides of the identities  $\sigma_\mu$, $\sigma_1$ and $\sigma_2$ are e.e. to each other and will be often used without a reference.

 \begin{fact} \label{xy3}

 \cite[Fact 3.2]{OS1} $xtxyty \sim xtyxty$,  $xytxty \sim yxtxty$,  $xtytxy \sim xtytyx$.

\end{fact}

Here is another very useful fact about the left and right sides of the identities $\sigma_\mu$, $\sigma_1$ and $\sigma_2$.

\begin{fact} \label{xtx} \cite[Fact 3.1]{OS1}
If $xtx$ is an isoterm for a monoid $S$, then

(i) the words $xtyxty$ and $xtxyty$ can only
form an identity of $S$ with each other;

(ii) the words $xytxty$ and $yxtxty$ can only
form an identity of $S$ with each other;

(iii) the words $xtytxy$ and $xtytyx$ can only
form an identity of $S$ with each other.
\end{fact}

\subsection{How to check that a monoid of the form $S(W)$ satisfies a balanced identity}

We use the word {\em substitution} to refer to the homomorphisms of the free semigroup and of the free monoid. Since every substitution $\Theta$ is uniquely determined by its values on the letters of the alphabet $\mathfrak A$, we write $\Theta: \mathfrak A \rightarrow \mathfrak A^+$ if $\Theta$ is a homomorphism of the free semigroup $\mathfrak A^+$ and we write $\Theta: \mathfrak A \rightarrow \mathfrak A^*$ if $\Theta$ is a homomorphism of the free monoid $\mathfrak A^*$.

If $\mathfrak X$ is a set of variables then we write ${\bf u}(\mathfrak X)$ to refer to the word obtained from ${\bf u}$ by deleting all occurrences of all variables that are not in $\mathfrak X$.
 If $\mathfrak X = \{y_1, \dots, y_k\} \cup \mathfrak Y$ for some variables $y_1, \dots, y_k$ and a set of variables $\mathfrak Y$ then instead of ${\bf u}(\{y_1, \dots, y_k\} \cup \mathfrak Y)$ we simply write ${\bf u}(y_1, \dots, y_k, \mathfrak Y)$.
We say that a pair of variables $\{x,y\}$ is {\em stable} in an identity $\bf u \approx \bf v$ if ${\bf u}(x,y)={\bf v}(x,y)$. Otherwise, we say that  $\{x,y\}$ is {\em unstable} in $\bf u \approx \bf v$. An identity ${\bf u} \approx {\bf v}$ is called {\em balanced} if every variable appears the same number of times in ${\bf u}$ and ${\bf v}$. If a semigroup $S$ satisfies all identities in a set $\Sigma$ then we write $S \models \Sigma$.

\begin{lemma} \label{nfbcombinations} \cite[Corollary 2.5]{OS3} Let $L=L^{\le}$ and $N$ be sets of words and ${\bf u} \approx {\bf v}$ be a balanced identity.
Let $W \subseteq L$ be such that $W^\preceq \cap N = \emptyset$.

Suppose that for every pair of variables $\{x, y\}$ unstable in ${\bf u} \approx {\bf v}$ and every substitution $\Theta: \mathfrak A \rightarrow \mathfrak A ^*$  such that $\Theta(x)$ contains some $a \in \mathfrak A$ and $\Theta(y)$ contains $b\ne a$, each of the following conditions is satisfied.

(i) If  $\Theta({\bf u}) \in L$ then $\Theta({\bf u}) \preceq {\bf n}$  for some ${\bf n} \in N$.

(ii) If  $\Theta({\bf v}) \in L$ then $\Theta({\bf v}) \preceq {\bf n}$  for some ${\bf n} \in N$.

Then $S(W) \models {\bf u} \approx {\bf v}$.

\end{lemma}

The following auxiliary statement illustrates how to use Lemma \ref{nfbcombinations}.

\begin{lemma} \label{axil}
 $S(\iso(\sigma_1)) \models xytxy \approx xytyx$ and $S(\iso(\sigma_2)) \models xytxy \approx yxtxy$.
\end{lemma}

\begin{proof} Take $L=\mathfrak A^*$ and $N =\{xytxty, yxtxty\}$.
Notice that $\{x,y\}$ is the only unstable pair of variables in the identity $xy t xy \approx yx t xy$.
 Let $\Theta: \mathfrak A \rightarrow \mathfrak A ^*$  be a substitution such that $\Theta(x)$ contains some $a \in \mathfrak A$ and $\Theta(y)$ contains $b\ne a$.
 If $\Theta(x)$ is not a power of $a$ and $\Theta(y)$ is a power of $b$ then $\Theta({xytxy}) \preceq xytxty$ and $\Theta({yxtxy}) \preceq xytxty$. If $\Theta(x)=a^k$ for some $k>0$ and $\Theta(y)=b^p$ for some $p>0$ then $\Theta({xytxy}) \preceq abtatb$ and $\Theta({yxtxy}) \preceq batatb$.
 So, Lemma \ref{nfbcombinations} implies that $S(\iso(\sigma_1)) \models xytxy \approx xytyx$. Dually, $S(\iso(\sigma_2)) \models xytxy \approx yxtxy$.
\end{proof}

\section{Sufficiency of each condition  in Theorem \ref{main1}}\label{sec:fb}

\subsection{Condition (A) in Theorem \ref{main1}}

For each $n>0$ denote $\mathcal A_ n = \{t_1xt_2x \dots t_{n}x \approx x^{n}t_1t_2 \dots t_{n}, x^{n} \approx x^{n+1}\}$.
It is easy to see that $\iso(\mathcal A_{n+1})$ is the set of all $n$-limited words in $\mathfrak A^*$.
Corollary 3.2 in \cite{JS} implies that for each $k>0$, $\iso(\mathcal A_{k})$ is finitely based by $\mathcal A_{k}^\delta$. Using the fact that modulo renaming variables, every 2-limited word over  a two-letter alphabet is a subword of one of the words in   $\{a^2b^2, abab, abba\}$, it is easy to show that  $\iso(\mathcal A_{3}) = \{a^2b^2, abab, abba\}^{\preceq}$.
 Thus we have the following.

\begin{lemma} \label{nlimited}

(i)  $\{a^2b^2, abab, abba\}$ is finitely based by $\mathcal A_3^\delta = \{t_1xt_2xt_3x \approx x^3t_1t_2t_{3}, x^3 \approx x^4\}^\delta$.

(ii) Given a set of words  $W$ we have $W \sim \{a^2b^2, abab, abba\}$ if and only if every word in $W$ is 2-limited and $W \preceq \{a^2b^2, abab, abba\}$.

\end{lemma}

Part (i) of Lemma \ref{nlimited} confirms that Condition (A) in Theorem \ref{main1} is sufficient for a set of words $W$ to be FB and Part (ii) provides an algorithm for checking when $W$ satisfies Condition (A).

\subsection{Condition (B) in Theorem \ref{main1}}

Part (i) of the following lemma confirms that Condition (B) in Theorem \ref{main1} is sufficient for a set of words $W$ to be FB and Part (ii) provides an algorithm for checking when $W$ satisfies Condition (B).

\begin{lemma} \label{aabb}

(i) $\{a^2b^2\}$ is finitely based by $\mathcal A_3^\delta \cup \{\sigma_1, \sigma_2\}^\delta$.

(ii) For a set of 2-limited words $W$ we have $W\sim \{a^2b^2\}$ if and only if $W \preceq x^2y^2$ and neither $xytxty$ nor $xtytxy$ is  an isoterm for $S(W)$.

\end{lemma}

\begin{proof}
Part (i) can be easily verified directly or using Theorem 4.4 in \cite{OS2}.

(ii)  Since $S(\{a^2b^2\})$ is finitely based by $\mathcal A_3^\delta \cup \{\sigma_1, \sigma_2\}^\delta$
we have \[\iso(\mathcal A_3 \cup \{\sigma_1, \sigma_2\}) = \{a^2b^2\}^{\preceq} .\]
Now if $W\sim \{a^2b^2\}$ then neither $xytxty$ nor $xtytxy$ is  an isoterm for $S(W)$ because  $xytxty$ is the left side of the identity
$\sigma_1$ and $xtytxy$ is the left side of the identity $\sigma_2$.

Conversely,  if $W \preceq x^2y^2$ and neither $xytxty$ nor $xtytxy$ is  an isoterm for $S(W)$ then $W \preceq xtx$.
Therefore,  $S(W) \models \{\sigma_1, \sigma_2\}$ by  Fact \ref{xtx}.
\end{proof}

\subsection{Condition (C)  in Theorem \ref{main1}}

We use $_{i{\bf u}}x$ to refer to the $i^{th}$ from the left occurrence of variable $x$ in a word ${\bf u}$. We use $_{last{\bf u}}x$ to refer to the last occurrence of $x$ in ${\bf u}$.
If the word $\bf u$ is clear from the context then we simply write $_{i}x$ or $_{last}x$.
Recall from the introduction that a set of words $W$ is called  hereditarily finitely based (HFB) if every monoid subvariety of $\var S(W)$ is finitely based. The following lemma provides a simple algorithm for recognizing HFB sets of words.

\begin{lemma} \label{hfb} \cite[Corollary 8.3]{OS3}
A set of words is HFB if and only if it is a subset of one of the following:

(i)  $\operatorname{Isot}(\sigma_1, \sigma_\mu)$ is the set of all words $\bf u \in \mathfrak A^*$ such that every adjacent pair of occurrences of two non-linear variables $x \ne y$ in ${\bf u}$ is of the form $\{{_{last{\bf u}}x}, {_{last{\bf u}}y} \}$;

(ii) $\operatorname{Isot}(\sigma_2, \sigma_\mu)$ is the set of all words $\bf u \in \mathfrak A^*$ such that every adjacent pair of occurrences of two non-linear variables $x \ne y$ in ${\bf u}$ is of the form $\{{_{1{\bf u}}x}, {_{1{\bf u}}y} \}$;

(iii) $\operatorname{Isot}(\sigma_1, \sigma_\mu, \sigma_2) = \operatorname{Isot}(\sigma_1, \sigma_\mu) \cap \operatorname{Isot}(\sigma_2, \sigma_\mu)$ is the set of all block-1-simple words.

\end{lemma}

The next lemma will be useful for proving the `only if' part of Theorem \ref{main1}.

\begin{lemma} \label{hfb1}
 A set of 2-limited words $W$ is HFB if and only if either $atbba \preceq W$ or  $abbta \preceq W$.
\end{lemma}

\begin{proof}  It is proved in \cite{LZL} that the word $atbba$ is finitely based by  $\mathcal A_3^\delta \cup \{\sigma_1, \sigma_\mu, xytxy \approx xy t yx \}^\delta$. Since by Lemma \ref{axil}, $S(\iso(\sigma_1)) \models xytxy \approx xytyx$, we have
\[\iso(\mathcal A_3 \cup \{\sigma_1, \sigma_\mu\}) = \{atbba\}^{\preceq},\] and dually
\[\iso(\mathcal A_3 \cup \{\sigma_2, \sigma_\mu\}) = \{abbta\}^{\preceq}.\]
In view of Lemma \ref{hfb}, a set of 2-limited words $W$ is HFB if and only if $W$ is a subset of $\iso(\mathcal A_3 \cup \{\sigma_1, \sigma_\mu\}) = \{atbba\}^{\preceq}$ or  $\iso(\mathcal A_3 \cup \{\sigma_2, \sigma_\mu\}) = \{abbta\}^{\preceq}$.
\end{proof}

It follows from Theorem 1.1 in \cite{EL} that every monoid which satisfies $\{\sigma_1, \sigma_\mu\}$ (or dually, $\{\sigma_2, \sigma_\mu\}$) is FB by its identities with at most two non-linear variables (see also Theorem 3.5 in \cite{OS1}). Therefore, the varieties
$\var S(\{atbba\})$ and $\var S(\{abbta\})$ have only finitely many subvarieties. In particular, modulo equational equivalence, there are only finitely many HFB sets of 2-limited words.

\subsection{Conditions (D) and (D$'$) in Theorem \ref{main1}}

 \begin{fact} \label{xy4}
 \cite[Fact 5.2]{OS2} $\{xytxy, xytyx\} \prec xyztxzy \sim yzxtzyx$.

\end{fact}

\begin{lemma} \label{abtab1} \cite[Theorem 5.3]{OS2} Let $S$ be a  monoid such that $S \models  \mathcal A_3 \cup \{ \sigma_{\mu}, yxxty  \approx xxyty\} = \Omega$. If the word $xyztxzy$ is an isoterm for $S$ then $S$ is finitely based by a subset of $\Omega^\delta \cup \{ytyxx \approx ytxxy, xxt \approx txx, xytxy \approx yxtyx, x^2 \approx x^3\}^\delta$.
\end{lemma}

If $W_1 \subseteq W_2$ are sets of words then we use $[W_1, W_2]$ to refer to the interval  between $\var S(W_1)$ and
$\var S(W_2)$ in the lattice of all semigroup varieties. Fact \ref{xy4},  Lemma \ref{abtab1} and the dual of Lemma \ref{abtab1} readily  imply the following.

\begin{cor} \label{D1} Every monoid in the interval $[\{abctacb\}, \{abtab, abtba, atbba\}]$ or dually, in the interval $[\{abctacb\}, \{abtab, abtba, abbta \}]$ is FB.
\end{cor}

Notice that  the intervals in Corollary \ref{D1} contain only finitely many monoid varieties, because in view of Lemma \ref{abtab1}, there are only finitely many possibilities for their identity bases. In view of  Fact \ref{xy4}, the interval  $[\{abtab, abtba\}, \{abtab, abtba, atbba\}]$ is a subinterval of $[\{abctacb\}, \{abtab, abtba, atbba\}]$ and dually,  the interval

$[\{abtab, abtba\}, \{abtab, abtba, abbta\}]$ is a subinterval of $[\{abctacb\}, \{abtab, abtba, abbta\}]$.
Thus Corollary \ref{D1} implies that Conditions (D) and (D$'$) in Theorem \ref{main1} are sufficient for a set of words $W$ to be FB.
The following lemma provides an algorithm for checking when $W$ satisfies Condition (D) or  (D$'$).

\begin{lemma}\label{abtab}
 A  set of 2-limited words $W$ belongs to one of the intervals

$[\{abtab, abtba\}, \{abtab, abtba, atbba\}]$ or $[\{abtab, abtba\}, \{abtab, abtba, abbta \}]$ if and only if $W \preceq \{xy t xy, xy t yx \}$, $W \not \preceq  xtxyty$ and  $W \not \preceq \{yxxty, ytxxy\}$.

\end{lemma}

\begin{proof}   Lemma \ref{abtab1} implies that the set $\{abtab, abtba, atbba\}$ is finitely based by $\mathcal A_3^\delta \cup \{\sigma_\mu, xxyty \approx yxxty\}^\delta$. Dually, the set $\{abtab, abtba, abbta\}$ is finitely based by $\mathcal A_3^\delta \cup \{\sigma_\mu, ytyxx \approx ytxxy\}^\delta$. Therefore, we have
 \begin{equation} \label{e1} \iso(\mathcal A_3 \cup \{\sigma_\mu, xxyty \approx yxxty\}) = \{abtab, abtba, atbba\}^{\preceq};\end{equation}
\[\iso(\mathcal A_3 \cup \{\sigma_\mu, ytyxx \approx ytxxy\}) = \{abtab, abtba, abbta\}^{\preceq}.\]

So, if $W$ belongs to one of these intervals then $xtxyty$ is not an isoterm for $S(W)$ because  $xtxyty$ is the left side of $\sigma_\mu$.
Also, one of the words $\{yxxty, ytxxy\}$ is not an isoterm for $S(W)$ because  $yxxty$ is the right side of  $xxyty \approx yxxty$ and
$ytxxy$ is the right side of  $ytyxx \approx ytxxy$.

Conversely,  let $W$ be a set of 2-limited words such that $W \preceq \{xy t xy, xy t yx \}$ but $W \not \preceq  xtxyty$ and  $W \not \preceq yxxty$. Since  $W \not \preceq xtxyty$,  Fact \ref{xtx} implies that $S(W) \models \sigma_\mu$.

\begin{claim}  $S(W) \models xxyty \approx yxxty$.\end{claim}

\begin{proof} We use Lemma  \ref{nfbcombinations} with $L = \mathfrak A^*$ and $N = \{  xtxyty,   yxxty \}$.
Notice that $\{x,y\}$ is the only unstable pair of variables in the identity $yxxty \approx xxyty$.  Let  $\Theta: \mathfrak A \rightarrow \mathfrak A^*$ be a substitution such that  $\Theta(x)$ contains some letter $a$ and $\Theta(y)$ contains $b\ne a$.
If the word $\Theta(x)$ is not a power of $a$, then we have  $\Theta(yxxty) \preceq xtxyty$ and  $\Theta(xxyty) \preceq xtxyty$.

So, we may assume that $\Theta(x)=a^k$ for some $k>0$.
Then $\Theta(yxxty)$ contains a subword $caa{\bf D}c$ for some variable $c \ne a$ and possibly empty word $\bf D$.
Consequently, $\Theta(yxxty) \preceq yxxty$. Let $d$ be the first letter in $\Theta(y)$ other than $a$. Then the word $\Theta(xxyty)$ contains a subword $(_ia) ({_1d})$ such that $i>1$. Therefore, we have $\Theta(xxyty) \preceq xtxyty$.
Lemma \ref{nfbcombinations} implies that $S(W) \models xxyty \approx yxxty$.
\end{proof}

Since every word in $W$ is 2-limited,  $S(W) \models \mathcal A_3$. Since $S(W) \models \mathcal A_3 \cup \{\sigma_\mu, xxyty \approx yxxty\}$, $S(W) \in \var S(\{abtab, abtba, atbba\})$ by \eqref{e1}.  Since $W \preceq \{xy t xy, xy t yx \}$, the set
 $W$ belongs to the interval $[\{abtab, abtba\}, \{abtab, abtba, atbba\}]$.

Assuming that  $W \not \preceq  ytxxy$ and using similar arguments we conclude that $W$ belongs to the interval $[\{abtab, abtba\}, \{abtab, abtba, abbta \}]$.
\end{proof}

\section{NFB intervals}

As in \cite{JS}, the words $x_1x_2 \dots x_n$ and $x_nx_{n-1}\dots x_1$ are denoted by $[Xn]$ and $[nX]$ respectively.
 We use ${\bf U}^t$ ($^t{\bf U}$) to denote the word obtained from  a word $\bf U$ by inserting a linear variable after (before) each occurrence of each variable in $\bf U$. For example, $[Zn]^t = z_1tz_2t \dots tz_nt$.

\begin{lemma}  \cite[Theorem 4.4]{OS1} \label{nfbsuf} For every monoid $S$ the following is true:

 (i) (Row 1 in Table 1 in \cite{OS1}) If the word $xyyx$ is an isoterm for $S$ and for each $n>1$, $S$ satisfies the identity  $xx[Yn][nY] \approx [Yn][nY]xx$, then $S$ is NFB;

(ii) (Row 2 in Table 1  in \cite{OS1}) If the words $\{yxxty, ytxxy\}$ are isoterms for $S$ and for each $n>1$, $S$ satisfies the identity ${\bf U}_n \approx {\bf V}_n$ in row 1 of Table \ref{classes}, then $S$ is NFB;

(iii) (Row 6 in Table 1 in \cite{OS1}) If the words $\{xtxyty, xytxy, xytyx\}$ are isoterms for $S$ and for each $n>1$, $S$ satisfies the identity ${\bf U}_n \approx {\bf V}_n$ in row 2 of Table \ref{classes}, then $S$ is NFB;

(iv) (Row 7 in Table 1 in \cite{OS1}) If the words $\{xtxyty, xyyx\}$ are isoterms for $S$ and for each $n>1$, $S$ satisfies the identity ${\bf U}_n \approx {\bf V}_n$ in row 3 of Table \ref{classes}, then $S$ is NFB.

\end{lemma}

\begin{theorem} \label{xyyx} Let $W$ be a set of words such that $W \preceq xyyx$ but $W \not \preceq xtxyty$ then $W$ is NFB.
\end{theorem}

\begin{proof} First, let us show that $S \models  xx[Yn][nY] \approx [Yn][nY]xx$.
To this aim we use Lemma \ref{nfbcombinations} with $L=\mathfrak A^*$ and $N =\{xtxyty\}$.

 Fix some $n> 1$. The only unstable pairs of variables  in $xx[Yn][nY] \approx [Yn][nY]xx$ are  $\{x,y_i\}$, $i=1, \dots, n$. Fix some  $1 \le i \le n$ and a substitution $\Theta: \mathfrak A \rightarrow \mathfrak A^*$ such that $\Theta(x)$ contains some letter $a$ and $\Theta(y_i)$ contains $b\ne a$.
Since $xx[Yn][nY]$ contains only non-linear variables, $\Theta(xx[Yn][nY])$ contains only non-linear variables.
Then  $\Theta(xx[Yn][nY])$ contains a subword $(_ja) ({_1c})$ for some non-linear variable $c \ne a$ and some $j>1$.
 Consequently, we have  $\Theta(xx[Yn][nY]) \preceq xtxyty$.
 By symmetric arguments, we have  $\Theta([Yn][nY]xx) \preceq xtxyty$.

  Lemma \ref{nfbcombinations} implies that for each $n>1$, the monoid $S(W)$ satisfies the identity  $xx[Yn][nY] \approx [Yn][nY]xx$. Therefore, $S(W)$ is NFB by Lemma \ref{nfbsuf}(i).
\end{proof}

For each $n>0$ we use $A_n$ to denote the set which contains $xt_1xt_2 \dots t_{n-1} x$ and all those words which are obtained
from  $xt_1xt_2 \dots t_{n-1} x$ by deleting some of the variables in $\{t_1, \dots, t_{n-1}\}$. For example, $A_3 = \{xt_1xt_2x, xxt_2x, xt_1xx, x^3\}$. The following fact can be easily verified and will be used without a reference.

\begin{fact} For $k>0$ and a set of words $W$ the following are equivalent:

(i) each word in $W$ is $k$-limited;

(ii) $S(W) \models \mathcal A_{k+1}$;

(iii) no word in $A_{k+1}$ is an isoterm for $S(W)$.

\end{fact}

\begin{theorem} \label{t1} Take sets of words $I$ and $N$ from one of the four rows in  Table \ref{classes}.
Let $W$ be a set of block-2-simple words such that $W \preceq I$ but $W \not \preceq {\bf n}$ for any ${\bf n} \in N$. Then
$W$ is NFB.
\end{theorem}

\begin{proof} Each time we use Lemma \ref{nfbcombinations} we take $L$ to be the set of all block-2-simple words. Evidently, this set of words is closed under taking subwords.

 {\bf Row 1 in Table \ref{classes}.}  Here $I = \{yxxty, ytxxy\}$ and $N = \{xtxyty\} \cup \{xy^mx \mid m>1\}$.

Fix some $n>1$. The only unstable pair  of variables  in $[Zn]^tyxx[Zn]y \approx [Zn]^t xxy [Zn]y$ is $\{x,y\}$.  Let $\Theta: \mathfrak A \rightarrow \mathfrak A^*$ be a substitution such that $\Theta(x)$ contains some letter $a$ and $\Theta(y)$ contains $b\ne a$.
If $\Theta(x)$ is not a power of $x$ then  $\Theta([Zn]^tyxx[Zn]y) \preceq xtxyty$ and  $\Theta([Zn]^t xxy [Zn]y) \preceq xtxyty$.
So, we may assume that $\Theta(x)$ is a power of $x$.

If  $\Theta([Zn]^tyxx[Zn]y)$ is a block-2-simple word then it contains a subword $ba^mb$ for
some $m>1$. Therefore, we have  $\Theta([Zn]^tyxx[Zn]y) \preceq xy^mx$ for some $m>1$.
If $\Theta([Zn]^t xxy [Zn]y)$ is a block-2-simple word then it contains a subword $({_ia}) ({_jb})$ where ${_ia}$ is not the first occurrence of $a$ and ${_jb}$ is not the last occurrence of $b$.
Consequently, we have  $\Theta([Zn]^t xxy [Zn]y) \preceq xtxyty$.

Lemma \ref{nfbcombinations} implies that for each $n>1$, the monoid $S(W)$ satisfies the identity ${\bf U}_n \approx {\bf V}_n$ in Row 1 of  Table \ref{classes}. Therefore, $S(W)$ is NFB by Lemma \ref{nfbsuf}(ii).

{\bf Row 2 in Table \ref{classes}.}  Here $I = \{xtxyty, xytxy, xytyx\}$ and $N = A_4 \cup \{xyxy\}$.

 Fix some $n>1$. The only unstable pair of variables  in $xy[Zn]xyt[nZ] \approx yx[Zn]yxt[nZ]$ is $\{x,y\}$.
  Let $\Theta: \mathfrak A \rightarrow \mathfrak A^*$ be a substitution such that $\Theta(x)$ contains some letter $a$ and $\Theta(y)$ contains letter $b\ne a$.
 If some variable occurs in $\Theta(xy[Zn]xyt[nZ])$ more than three times, then $\Theta(xy[Zn]xyt[nZ]) \preceq {\bf u}$ for some ${\bf u} \in A_4$.
 If $\Theta(xy[Zn]xyt[nZ])$ is a 3-limited block-2-simple word then $\Theta(x)=a$, $\Theta(y)=b$ and $\Theta(([Zn])$ is the empty word. Therefore, $\Theta(xy[Zn]xyt[nZ]) \preceq xyxy$.

Lemma \ref{nfbcombinations} implies that for each $n>1$, the monoid $S(W)$ satisfies the identity ${\bf U}_n \approx {\bf V}_n$ in Row 2 of  Table \ref{classes}. Therefore, $S(W)$ is NFB by Lemma \ref{nfbsuf}(iii).

{\bf Row 3 in Table \ref{classes}.}  Here $I = \{xtxyty, xyyx\}$ and $N = A_4 \cup \{xxyy\}$.

 Fix some $n>1$. Each unstable pairs of variables  in $[Xn][nX][Yn][nY] \approx [Yn][nY][Xn][nX]$ is of the form $\{x_i,y_j\}$ for some $1 \le i,j \le n$.
  Let $\Theta: \mathfrak A \rightarrow \mathfrak A^*$ be a substitution such that
 $\Theta(x)$ contains some letter $a$ and $\Theta(y)$ contains letter $b\ne a$.  If some variable occurs in $\Theta([Xn][nX][Yn][nY])$ more than three times then $\Theta([Xn][nX][Yn][nY]) \preceq {\bf u}$ for some ${\bf u} \in A_4$.
 If $\Theta([Xn][nX][Yn][nY])$ is a 3-limited block-2-simple word then
 $\Theta(x)=a$, $\Theta(y)=b$ and  the value of $\Theta$ on all other letters  the empty word. Therefore,  $\Theta([Xn][nX][Yn][nY]) \preceq xxyy$.

Lemma \ref{nfbcombinations} implies that for each $n>1$, the monoid $S(W)$ satisfies the identity ${\bf U}_n \approx {\bf V}_n$ in Row 3 of  Table \ref{classes}. Therefore, $S(W)$ is NFB by Lemma \ref{nfbsuf}(iv).

{\bf Row 4 in Table \ref{classes}.}  Here $I = \{xtxyty, xytxy, xytyx\}$ and $N = \{xyxyx\} \cup \{xy^mx \mid m>1\}$ and
this is proved in Theorem 4.4 (row 6 in Table 1) in \cite{OS3}.
\end{proof}

\begin{table}[tbh]
\begin{center}
\small
\begin{tabular}{|l|l|l|}
\hline set $I$ &identity  ${\bf U}_n \approx {\bf V}_n$ for $n>1$ & set $N$\\
\hline

\protect\rule{0pt}{10pt}$yxxty$, $ytxxy $&$[Zn]^t \hskip .02in yxx[Zn]y \approx [Zn]^t \hskip .02in xxy[Zn]y $& $\{xtxyty\} \cup \{xy^mx| m>1\}$ \\
\hline

\protect\rule{0pt}{10pt}$xtxyty$, $xytxy$, $xytyx$    & $xy[Zn]xyt[nZ] \approx yx[Zn]yxt[nZ]$  &  $A_4 \cup \{xyxy\}$ \\
\hline

\protect\rule{0pt}{10pt}$xtxyty$, $xyyx$    & $[Xn][nX][Yn][nY] \approx [Yn][nY][Xn][nX]$  &  $A_4 \cup \{xxyy\}$\\
\hline

\protect\rule{0pt}{10pt}$xtxyty$, $xytxy$, $xytyx$ & $xy[An]yxt[nA]  \approx yx[An]xyt[nA]$&$xyxyx, \{xy^mx| m>1\}$ \\
\hline

\end{tabular}
\caption{Four NFB intervals $[I, \iso(L, \Sigma)]$\protect\rule{0pt}
{11pt}}

\label{classes}
\end{center}
\end{table}

\begin{lemma} \label{fb2} \cite[Corollary 4.5]{OS3} Let $W$ be a set of block-2-simple words such that $W\preceq \{xytxty, xtytxy\}$.
Then either $W$ is NFB or $W \preceq \{xytxy, xytyx\}$.
\end{lemma}

\begin{lemma} \label{fb1l}  \cite[Corollary 4.7]{OS3} Let $W$ be a set of block-2-simple words such that $W \preceq xtxyty$ but one of the words $\{xytxty, xtytxy\}$ is not an isoterm for $S(W)$. Then either $W$ is NFB or $W \preceq xxyy$ and both words $\{xytxty, xtytxy\}$ are not isoterms for $S(W)$.
\end{lemma}

\section{An algorithm that recognizes FB sets of words among sets of 2-limited block-2-simple words} \label{sec:alg}

Using Lemma \ref{hfb1} and  quasi-order $\preceq$ on sets of words introduced in Section \ref{sec:iso}, Theorem \ref{main1} can be reformulated as follows.

\begin{theorem} \label{main} A set of 2-limited block-2-simple words $W$ is FB if and only if

 $W \sim \{a^2b^2, abab, abba\}$ or $W \sim \{a^2b^2\}$ or

 $\{atbba\} \preceq W$ or $\{abbta\} \preceq W$ or

$\{abtab, abtba, atbba\} \preceq W \preceq \{abtab, abtba\}$ or

$\{abtab, abtba, abbta\} \preceq W \preceq \{abtab, abtba\}$.

\end{theorem}

\begin{proof}
Let $W$ be a set of 2-limited block-2-simple words and consider two cases.

{\bf Case 1:} $W \not \preceq xtxyty$.

In this case Theorem \ref{xyyx} implies that either $W$ is NFB or $W \not \preceq xyyx$. If $W \not \preceq xyyx$, then Theorem \ref{t1}(row 1 in Table \ref{classes}) and the fact that each word in $W$ is 2-limited implies that either $W$ is NFB or  $W \not\preceq \{yxxty, ytxxy\}$. So, we may assume that $W \not\preceq \{yxxty, ytxxy\}$.

If $W \preceq \{xytxy, xytyx\}$ then  Lemma \ref{abtab}  implies that either
\[\{abtab, abtba, atbba\} \preceq W \preceq \{abtab, abtba\}\] or
\[\{abtab, abtba, abbta\} \preceq W \preceq \{abtab, abtba\}.\]
Thus $W$ is FB by Corollary \ref{D1}.

If $W \not \preceq \{xytxy, xytyx\}$ then Lemma \ref{fb2} implies that either $W$ is NFB or $W \not \preceq \{xytxty, xtytyx\}$.

If $W \not \preceq \{xytxty, xtytyx\}$ then Lemma \ref{hfb1} implies that $W$ is HFB  and
$\{atbba\} \preceq W$ or $\{abbta\} \preceq W$.

{\bf Case 2:} $W \preceq xtxyty$.

 If $W \not\preceq \{xytxty, xtytxy\}$, then Lemma \ref{fb1l} implies that either $W$ is NFB or $W \preceq xxyy$ and both words $\{xytxty, xtytxy\}$ are not isoterms for $S(W)$.
 In the later case Lemma \ref{aabb} implies that $W$ is FB and $W \sim \{a^2b^2\}$.

 If $W \preceq \{xytxty, xtytxy\}$, then Lemma \ref{fb2} implies that either $W$ is NFB or $W \preceq \{xtxyty, xytxy, xytyx\}$.
Then Theorem \ref{t1}(row 2 in Table \ref{classes}) implies that  $W$ is either NFB or $W \preceq xyxy$.
Also, Theorem \ref{t1}(row 4 in Table \ref{classes})) implies that $W$ is either NFB or $W \preceq xyyx$.
If $W \preceq xyyx$, then by Theorem \ref{t1}(row 3 in Table \ref{classes}) we have that $W$ is either NFB or $W \preceq xxyy$. If $W \preceq \{xyxy, xyyx, xxyy\}$ then Lemma \ref{nlimited} implies that $W$ is FB and $W \sim \{aabb, abab, abba\}$.
\end{proof}

\section {Some sufficient conditions under which a word is e.e. to a set of words with at most two non-linear variables} \label{sec:twovar}

 According to Example 5.4 in \cite{OS2}  the monoid $S(\{abctacb\})$ is FB and is contained in  $\var \{\sigma_\mu \}$. But $\{abctacb\}$
is not e.e. to  any set of block-2-simple words, because it does not satisfy any of the conditions of Theorem \ref{main1}.

In this section, we show that every HFB set of words is e.e. to a set of words with at most two non-linear variables (cf. Theorem \ref{eqeq}).
We also show that every FB set of  block-2-simple words $W$ such that $S(W)$ is contained in one of the varieties $\var \{\sigma_\mu\}$, $\var \{\sigma_1\}$ or $\var \{\sigma_2\}$ is e.e. to  a set of words with at most two non-linear variables (cf. Corollary \ref{eq2}).

First, we formulate a general sufficient condition under which a word is e. e. to a set of words with at most two non-linear variables.
The set of all linear variables in a word $\bf u$ is denoted by $\lin ({\bf u})$
and the set of all non-linear variables is denoted by $\non ({\bf u})$.

\begin{lemma}  \label{delblock}
Let $\bf u$ be a word such that

(i) each block of $\bf u$ is a product of powers of pairwise distinct variables;

(ii) if two distinct variables $x$ and $y$ are adjacent in some block of $\bf u$ then they are adjacent in each other block that contains both $x$ and $y$.

Let $B$ denote the set of unordered pairs of non-linear variables  defined as follows:
$\{x,y\} \in B$ if and only if some occurrences of $x$ and $y$ are adjacent in $\bf u$.
Then
\[{\bf u} \sim \bigcup_ {x \in \non({\bf u})} {\bf u}(x, \lin({\bf u})) \cup \bigcup_{\{x,y\} \in B} {\bf u}(x,y,\lin({\bf u})).\]

\end{lemma}

\begin{proof} According to Lemma 4.1 in \cite{JS}, if $\bf v$ is obtained by erasing a prefix (suffix) of a block of a word $\bf u$ then  $\{{\bf u}, xy\} \preceq {\bf v}$. This implies that
\[{\bf u} \preceq
\bigcup_ {x \in \non({\bf u})} {\bf u}(x, \lin({\bf u})) \cup \bigcup_{\{x,y\} \in B} {\bf u}(x,y,\lin({\bf u})).\]

Suppose now that each word in the set  \[\bigcup_ {x \in \non({\bf u})} {\bf u}(x, \lin({\bf u})) \cup \bigcup_{\{x,y\} \in B} {\bf u}(x,y,\lin({\bf u}))\] is an isoterm for a monoid $S$. Then each pair of variables in stable in $\bf u$ with respect to $S$ and consequently, the word $\bf u$ is an isoterm for $S$.
Therefore,
\[{\bf u} \succeq \bigcup_ {x \in \non({\bf u})} {\bf u}(x, \lin({\bf u})) \cup \bigcup_{\{x,y\} \in B} {\bf u}(x,y,\lin({\bf u})).\]
\end{proof}

\begin{theorem} \label{eqeq}
Every HFB word is e.e. to a finite set of words with at most two non-linear variables.
\end{theorem}

\begin{proof}  Let $\bf u$ be a HFB word.
Then modulo duality, every adjacent pair of occurrences of two non-linear variables $x \ne y$ in ${\bf u}$ is of the form $\{{_{1{\bf u}}x}, {_{1{\bf u}}y} \}$ by  Lemma \ref{hfb}.
Therefore, each block of $\bf u$ is a product of powers of pairwise distinct variables.
If some block $\bf B$ of $\bf u$ contains a subword $(_{1{\bf u}}x)({_{1{\bf u}}y})$ for some distinct variables $x$ and $y$, then any other block that contains $x$ or $y$ can only be either a power of $x$ or a power of $y$.
Therefore,  ${\bf u}$  is e.e. to a finite set of words with at most two non-linear variables  by Lemma  \ref{delblock}.
\end{proof}

\begin{ex}  $\{abct_1at_2bt_3c\} \sim \{abt_1at_2b \}$ is HFB.

\end{ex}

The next lemma is an immediate consequence of Theorem 3.5 in \cite{OS3}.

\begin{lemma} \label{firstsim} For a set of words $W$ the following is true:

(i) $S(W) \in \var \{\sigma_1, \sigma_2 \}$ iff none of the words $\{xytxty, xtytxy \}$ is an isoterm for $S(W)$ iff every adjacent pair of occurrences of two non-linear variables $x \ne y$ in each ${\bf u} \in W$ is of the form $\{{_{1{\bf u}}x}, {_{last{\bf u}}y} \}$;

(ii) $S(W) \in \var \{\sigma_\mu \}$ iff the word $xtxyty$ is not an isoterm for $S(W)$ iff every adjacent pair of occurrences of two non-linear variables $x \ne y$ in each ${\bf u} \in W$ is either of the form $\{{_{1{\bf u}}x}, {_{1{\bf u}}y} \}$ or of the form $\{{_{last{\bf u}}x}, {_{last{\bf u}}y} \}$.

\end{lemma}

\begin{theorem} \label{eqeq1}

(i) If $\bf u$ is a word such that ${\bf u} \not  \preceq xytxty$, ${\bf u} \not  \preceq xtxtxy$ and each non-linear variable is at least 3-occurring
in $\bf u$ then $\bf u$  is e.e. to a set of words with at most two non-linear variables.

(ii)  If $\bf u$ is a block-2-simple word such that ${\bf u} \not  \preceq xytxty$ and ${\bf u} \not  \preceq xtxtxy$  then $\bf u$  is e.e. to a set of words with at most two non-linear variables.

(iii) If $\bf u$ is a block-2-simple word such that ${\bf u} \not  \preceq xtxyty$
then $\bf u$ is e.e. to a set of words with at most two non-linear variables.

\end{theorem}

\begin{proof}

(i) By Lemma \ref{firstsim}, every adjacent pair of occurrences of two non-linear variables $x \ne y$ in ${\bf u}$ is of the form $\{{_{1{\bf u}}x}, {_{last{\bf u}}y} \}$. Since each non-linear variable occurs at least three times in $\bf u$, each block of $\bf u$ is a product of powers of pairwise distinct variables. (If 2-occurring variables were allowed then one could have ${\bf u} = [Yn]^t\hskip .02in xy_1z_1 \dots y_nz_nx \hskip .02in ^t[Zn]$).
Also, if two distinct variables $x$ and $y$ are adjacent in some block of $\bf u$ then no other block of $\bf u$ contains both $x$ and $y$.
Therefore,  ${\bf u}$  is e.e. to a finite set of words with at most two non-linear variables by Lemma  \ref{delblock}.

(ii) By Lemma \ref{firstsim}, every adjacent pair of occurrences of two non-linear variables $x \ne y$ in ${\bf u}$ is of the form $\{{_{1{\bf u}}x}, {_{last{\bf u}}y} \}$. Notice that for each pair of variables $x$ and $y$ the word $\bf u$ contains at most one block that contains both $x$ and $y$ and this block (if any)  is  $x^ny^m$  or $y^nx^m$  for some $n,m > 0$.  So, the statement follows from Lemma \ref{delblock}.

(iii) By Lemma \ref{firstsim}, every adjacent pair of occurrences of two non-linear variables $x \ne y$ in ${\bf u}$ is either of the form $\{{_{1{\bf u}}x}, {_{1{\bf u}}y} \}$ or of the form $\{{_{last{\bf u}}x}, {_{last{\bf u}}y} \}$. Then each block of $\bf u$ that is not a power of a variable  is either of the form $(_{1 \bf u}x) (y^n)$ or $(x^m) ({_{last{\bf u}}y})$ for some variables $x \ne y$ and $n,m>0$. So, the statement follows from Lemma \ref{delblock}.
\end{proof}

\begin{cor} \label{eq2}  Let $W$ be a set of block-2-simple words such that at least one of the words $\{ xtxyty, xytxty, xtytxy\}$ is not an
isoterm for $S(W)$.  If $W$ is FB then $W$ is e.e. to a set of words with at most two non-linear variables.
\end{cor}

\begin{proof} If $W \not \preceq xtxyty$ then
$W$ is e.e. to a set of words with at most two non-linear variables by Theorem \ref{eqeq1}.  If $W$ is a FB set such that $W \preceq xtxyty$  but $W \not \preceq \{xytxty, xtytxy\}$ then by Lemma \ref{fb1l} we have  $W \not \preceq xytxty$ and  $W \not \preceq xtytxy$. Therefore, the set $W$ is e.e. to a set of words with at most two non-linear variables by Theorem \ref{eqeq1}.
\end{proof}

\begin{cor} \label{twoletter}  Every finitely based set of 2-limited block-2-simple words is e.e. to a set of words with at most two non-linear variables.
\end{cor}

\begin{proof} Let $W$ be a FB set of 2-limited block-2-simple words. If one of the words $\{ xtxyty, xytxty, xtytxy\}$ is not an
isoterm for $S(W)$ then $W$ is e.e. to a set of words in at most two non-linear variables by Corollary \ref{eq2}.
If  $W \preceq \{ xtxyty, xytxty, xtytxy\}$ then the proof of  Theorem \ref{main} implies that  $W \sim \{aabb, abab, abba\}$.
\end{proof}

\section{Another sufficient condition under which an arbitrary set of words is NFB}  \label{sec:sufcon}

The goal of this section is to prove a simple sufficient condition under which an arbitrary set of words is NFB.
This sufficient condition is similar to the one in Theorem \ref{xyyx}.
It says that a set of words $W$ is NFB provided a certain word is an isoterm for $W$ and a certain word is not.
It seems that such simple sufficient conditions are rare.
First, we establish a new sufficient condition under which a monoid is NFB.

If $S$ is a semigroup and $\bf u$ is a word then we use $[\![{\bf u}]\!]_S$ to denote the equivalence class of the fully invariant congruence $\sim_S$ on $\mathfrak A^+$ containing $\bf u$.  The set $\ocs({\bf u}) = \{ {_{i{\bf u}}x} \mid x \in \mathfrak A, 1 \le i \le occ_{\bf u} (x) \}$ of all occurrences of all variables in ${\bf u}$ is called the {\em occurrence set of ${\bf u}$}. The word ${\bf u}$ induces a (total) order $<_{\bf u}$ on the set $\ocs({\bf u})$ defined by ${_{i{\bf u}}x} <_{\bf u} {_{j{\bf u}}y}$ if and only if the $i^{th}$ occurrence of $x$ precedes the $j^{th}$ occurrence of $y$ in ${\bf u}$.
The following lemma is a special case of Lemma 2.5 in \cite{OS1}.

\begin{lemma} \label{nfblemma} Let $S$ be a semigroup.
Suppose that for each $n$ large enough one can find a balanced identity ${\bf U}_n \approx {\bf V}_n$ of $S$ in at least $n$ variables
such for some variables $x$ and $z$ in  ${\bf U}_n$ the following is true:

(i) $({_{1{\bf U}_n}x}) <_{\bf U} ({_{1{\bf U}_n}z})$ but   $({_{1{\bf V}_n}z}) <_{\bf U} ({_{1{\bf V}_n}x})$;

(ii) if ${\bf U} \in [\![{\bf U}_n]\!]_S$ such that  $({_{1{\bf U}}x}) <_{\bf U} ({_{1{\bf U}}z})$
then for every identity ${\bf u} \approx {\bf v}$ of $S$ in less than $n/3$ variables and every substitution  $\Theta: \mathfrak A \rightarrow \mathfrak A^+$ such that $\Theta({\bf u})={\bf U}$ we have $({_{1{\bf V}}x}) <_{\bf V} ({_{1{\bf V}}z})$ where
${\bf V} = \Theta({\bf v})$.

Then $S$ is NFB.
\end{lemma}

 Let $\Theta: \mathfrak A \rightarrow \mathfrak A^+$ be a substitution such that $\Theta({\bf u})={\bf U}$. Then
$\Theta$ induces a map $\Theta_{\bf u}$ from $\ocs({\bf u})$
into  subsets of $\ocs({\bf U})$ as follows.
If $1 \le i \le occ_{\bf u}(x)$ then $\Theta_{\bf u}({_{i{\bf u}}x})$ denotes the set of all elements of $\ocs({\bf U})$
contained in the subword of ${\bf U}$ of the form $\Theta(x)$ that corresponds to the $i^{th}$ occurrence of variable $x$ in ${\bf u}$. For example, if $\Theta(x)=ab$ and $\Theta(y)=bab$ then $\Theta_{xyx}({_{2(xyx)}x})=\{{_{3(abbabab)}a}, {_{4(abbabab)}b} \}$.
Evidently, for each $d \in \ocs({\bf u})$ the set $\Theta_{\bf u} (d)$ is an interval in $(\ocs({\bf U}), <_{\bf U})$.
Now we define a function $\Theta^{-1}_{\bf u}$ from $\ocs({\bf U})$ to $\ocs({\bf u})$ as follows.
If $c \in \ocs({\bf U})$ then $\Theta^{-1}_{\bf u}(c)=d$ when $\Theta_{\bf u} (d)$ contains $c$.
For example, $\Theta^{-1}_{xyx}({_{3(abbabab)}a})= {_{2(xyx)}x}$.

\begin{theorem} \label{sufnfb} Let $S$ be a monoid which
satisfies
\[
{\bf U}_n = x z [Yn] p t z [nY] p x \approx  z [Yn] p x t xz [nY] p = {\bf V}_n.
\]

If $xytxy$ is an isoterm for $S$ then $S$ is NFB.
\end{theorem}

\begin{proof} Take ${\bf U} \in [\![{\bf U}_n]\!]_S$. Since $xtx$ is an isoterm for $S$, the identity ${\bf U}_n \approx {\bf U}$ is balanced.
Since $xytxy$ is an isoterm for $S$ the word ${\bf U}$ satisfies the following properties:

(P1) for each $1 \le i \le n$ we have ${\bf U} (z,y_i,p, t) = zy_iptzy_ip$;

(P2) for each $1 \le i, j \le n$, if  ${_{1{\bf U}}y_i} <_{\bf U} {_{1{\bf U}}y_j}$ then  ${_{2{\bf U}}y_j} <_{\bf U} {_{2{\bf U}}y_i}$;

(P3) if  ${_{1{\bf U}}x} <_{\bf U} {_{1{\bf U}}z}$ then  ${_{2{\bf U}}p} <_{\bf U} {_{2{\bf U}}x}$.

Fix $n>10$. Evidently,  $({_{1{\bf U}_n}x}) <_{\bf U} ({_{1{\bf U}_n}z})$ but   $({_{1{\bf V}_n}z}) <_{\bf U} ({_{1{\bf V}_n}x})$.
Thus $S$ satisfies the first condition of Lemma \ref{nfblemma}.

In order to verify the second condition of Lemma \ref{nfblemma}, take ${\bf U} \in [\![{\bf U}_n]\!]_S$ such that
${_{1{\bf U}}x} <_{\bf U} {_{1{\bf U}}z}$.
Let ${\bf u}$ be a word in less than $n/3$ variables such that $\Theta({\bf u})={\bf U}$ for some substitution $\Theta: \mathfrak A \rightarrow \mathfrak A^+$.
Since  ${\bf u}$ has less than $n/3$ variables, for some
$c \in \ocs({\bf u})$ and $1 < i < j < n$ both ${_{2{\bf U}}y}_{i}$ and ${_{2{\bf U}}y}_{j}$ are contained in $\Theta_{\bf u}(c)$.
Then Property (P2) implies that $c$ must be the only occurrence of a linear variable $t_1$ in ${\bf u}$.

Since $\Theta^{-1}_{\bf u}$ is a homomorphism from $(\ocs({\bf U}), <_{\bf U})$ to $(\ocs({\bf u}), <_{\bf u})$, Properties (P1)--(P3)
imply that
\[\Theta^{-1}_{\bf u}({_{1{\bf U}}x})  \le_{\bf u}  \Theta^{-1}_{\bf u}({_{1{\bf U}}z}) \le_{\bf u}  \Theta^{-1}_{\bf u}({_{1{\bf U}}p}) \le_{\bf u}    \Theta^{-1}_{\bf u}({_{{\bf U}}t}) \le_{\bf u}  \Theta^{-1}_{\bf u}({_{2{\bf U}}z}) \le_{\bf u} \]
\[  \le_{\bf u} (_{\bf u}t_1)  \le_{\bf u} \Theta^{-1}_{\bf u}({_{2{\bf U}}p})  \le_{\bf u} \Theta^{-1}_{\bf u}({_{2{\bf U}}x}).\]

Now $\Theta^{-1}_{\bf u}({_{1{\bf U}}x})$ and $\Theta^{-1}_{\bf u}({_{1{\bf U}}z})$ are the first occurrences of some variables $x$ and $z$ in ${\bf u}$. Thus $({_{1{\bf u}}x}) <_{\bf u} ({_{1{\bf u}}z})$.  If either $x$ or $z$ is linear in ${\bf u}$ then $({_{1{\bf v}}x}) <_{\bf v} ({_{1{\bf v}}z})$  because the word $xtx$ is an isoterm for $S$.
 If both $x$ and $z$ occur twice in ${\bf u}$ then  $_{1{\bf u}}x = \Theta^{-1}_{\bf u}({_{1{\bf U}}x})$, $_{2{\bf u}}x =\Theta^{-1}_{\bf u}({_{2{\bf U}}x})$, $_{1{\bf u}}z = \Theta^{-1}_{\bf u}({_{1{\bf U}}z})$ and $_{2{\bf u}}z =\Theta^{-1}_{\bf u}({_{2{\bf U}}z})$. In this case $({_{1{\bf v}}x}) <_{\bf v} ({_{1{\bf v}}z})$,
because ${\bf u}(x,z,t,t_1) = xztzt_1x$ is an isoterm for $S$.

Therefore,  $({_{1{\bf V}}x}) <_{\bf V} ({_{1{\bf V}}z})$ and $S$ is NFB  by Lemma \ref{nfblemma}.
\end{proof}

\begin{cor} \label{xytxy} Let $W$ be a set of words such that $W \preceq xytxy$ but $W \not \preceq xytyx$ then $W$ is NFB.
\end{cor}

\begin{proof} We use Lemma \ref{nfbcombinations} with $L=\mathfrak A^*$ and $N =\{xytyx\}$.

Fix some $n> 1$. The only unstable pairs of variables  in $x z [Yn] p t z [nY] p x \approx  z [Yn] p x t x z [nY] p$ are  $\{x,z\}$, $\{x,p\}$, $\{x,y_i\}$, $i=1, \dots n$. Let $q \in \{z, p, y_1, \dots, y_n\}$ and $\Theta: \mathfrak A \rightarrow \mathfrak A^*$ be a substitution such that $\Theta(x)$ contains some letter $a$ and $\Theta(q)$ contains $b\ne a$.

If $\Theta (z [Yn] p) = \Theta (z [nY] p) = {\bf C}$ then $\Theta (x z [Yn] p t z [nY] p x) = \Theta(x){\bf C} \Theta(t){\bf C}\Theta(x) \preceq xytyx$.
If $\Theta (z [Yn] p) \ne  \Theta (z [nY] p)$ then for some $1 \le i < j \le n$ we have $\Theta(y_i)\Theta(y_j) \ne \Theta(y_j)\Theta(y_i)$ and
$\Theta(y_{i+1}) = \dots = \Theta(y_{j-1}) = 1$. In this case we also have that

$\Theta (x z [Yn] p t z [nY] p x) \preceq xytyx$.
Similarly, we have
$\Theta (z [Yn] p x t xz [nY] p) \preceq xytyx$.

  Lemma \ref{nfbcombinations} implies that for each $n>1$, $S(W)$ satisfies $x z [Yn] p t z [nY] p x \approx  z [Yn] p x t x z [nY] p$.
Therefore, $S(W)$ is NFB by Theorem \ref{sufnfb}.\end{proof}

 Using the interval notation, Theorem \ref{xyyx} and Corollary \ref{xytxy}  can be reformulated as follows
(cf. Table \ref{intervals}).

\begin{cor}\label{sevenint} Every monoid in each of the following intervals is NFB:

(i) (Theorem \ref{xyyx}) $[\{xyyx\}, \iso(\sigma_\mu)]$;

(ii) (Corollary \ref{xytxy})  $[\{xytxy\}, \iso( xytyx \approx yx t xy)]$.

\end{cor}

\begin{table}[tbh]
\begin{center}
\small
\begin{tabular}{|l|l|l|}
\hline set $I$ &identity  ${\bf U}_n \approx {\bf V}_n$ for $n>1$ & set $N$\\
\hline

\protect  \rule{0pt}{10pt}$xyyx$ & $xx[Yn][nY] \approx [Yn][nY]xx$ &$xt_1xyt_2y$ \\
\hline

\protect\rule{0pt}{10pt}$xytxy$ & $xz[Yn]ptz[nY]px \approx z[Yn]pxtxz[nY]p$ & $xytyx$ \\
\hline

\end{tabular}
\caption{Two NFB intervals $[I, \iso(\mathfrak A^*, \Sigma)]$ \protect\rule{0pt}
{11pt}}

\label{intervals}
\end{center}
\end{table}

\subsection*{Acknowledgement} The author thanks Edmond Lee for helpful comments.

\end{document}